\documentclass[11pt]{amsart}
\usepackage[leqno]{amsmath}
\usepackage{amssymb,latexsym,soul,cite,amsthm,color,enumitem,graphicx,tikz,mathtools,accents}
\usepackage[colorlinks=true,urlcolor=cerulean,citecolor=cerulean,linkcolor=cerulean,linktocpage,pdfpagelabels,bookmarksnumbered,bookmarksopen]{hyperref}
\definecolor{cerulean}{rgb}{0.0, 0.48, 0.65}
\usepackage[english]{babel}
\usepackage[left=3cm,right=3cm,top=3cm,bottom=3cm]{geometry}

\numberwithin{equation}{section}

\newtheorem{theorem}{Theorem}[section]
\theoremstyle{plain}
\newtheorem{lemma}[theorem]{Lemma}
\theoremstyle{plain}
\newtheorem{proposition}[theorem]{Proposition}
\theoremstyle{plain}

\theoremstyle{definition}
\newtheorem{remark}[theorem]{Remark}

\newcommand{\N}{{\mathbb N}}

\newcommand{\R}{{\mathbb R}}
\newcommand{\eps}{\varepsilon}
\newcommand{\beq}{\begin{equation}}
\newcommand{\eeq}{\end{equation}}
\renewcommand{\le}{\leqslant}
\renewcommand{\ge}{\geqslant}

\newcommand{\w}{W^{s,p}_0(\Omega)}
\newcommand{\fpl}{(-\Delta)_p^s\,}
\newcommand{\ds}{{\rm d}_\Omega^s}
\newcommand{\p}{p^*_s}
\newcommand{\cs}{C_s^0(\overline{\Omega})}

\makeatletter
\newcommand{\leqnomode}{\tagsleft@true}
\newcommand{\reqnomode}{\tagsleft@false}
\makeatother

\newenvironment{enumroman}{\begin{enumerate}

}{\end{enumerate}}

\title[Logistic equation for the fractional $p$-Laplacian]{On the logistic equation for the fractional $p$-Laplacian}

\author[A.\ Iannizzotto, S.\ Mosconi, N.S.\ Papageorgiou]{Antonio Iannizzotto, Sunra Mosconi, and Nikolaos S.\ Papageorgiou}

\address[A.\ Iannizzotto]{Department of Mathematics and Computer Science
\newline\indent
University of Cagliari
\newline\indent
Via Ospedale 72, 09124 Cagliari, Italy}
\email{antonio.iannizzotto@unica.it}

\address[S.\ Mosconi]{Department of Mathematics and Computer Science
\newline\indent
University of Catania
\newline\indent
Viale A.\ Doria 6, 95125 Catania, Italy}
\email{mosconi@dmi.unict.it}

\address[N.S.\ Papageorgiou]{Department of Mathematics
\newline\indent
National Technical University
\newline\indent
Zografou Campus, Athens 15780, Greece}
\email{npapg@math.ntua.gr}


\begin{document}

\begin{abstract}
We consider a Dirichlet type problem for a nonlinear, nonlocal equation driven by the degenerate fractional $p$-Laplacian, with a logistic type reaction depending on a positive parameter. In the subdiffusive and equidiffusive cases, we prove existence and uniqueness of the positive solution when the parameter lies in convenient intervals. In the superdiffusive case, we establish a bifurcation result. A new strong comparison result, of independent interest, plays a crucial role in the proof of such bifurcation result.
\end{abstract}

\maketitle

\begin{center}
Version of \today\
\end{center}

\section{Introduction}\label{sec1}

\noindent
The present paper is devoted to the study of the following nonlinear elliptic equation of fractional order with Dirichlet type condition:
\[(P_\lambda) \qquad \begin{cases}
\fpl u = \lambda u^{q-1}-u^{r-1} & \text{in $\Omega$} \\
u > 0 & \text{in $\Omega$} \\
u = 0 & \text{in $\Omega^c$.}
\end{cases}\]
Here $\Omega\subset\R^N$ ($N\ge 2)$ is a bounded domain with $C^{1,1}$ boundary $\partial\Omega$, $s\in(0,1)$, $p\ge 2$ are s.t.\ $ps<N$, and the leading operator is the degenerate fractional $p$-Laplacian, defined for all $u:\R^N\to\R$ smooth enough and $x\in\R^N$ by
\[\fpl u(x)=2\lim_{\eps\to 0^+}\int_{\{|x-y|>\eps\}}\frac{|u(x)-u(y)|^{p-2}(u(x)-u(y))}{|x-y|^{N+ps}}\,dy\]
(which for $p=2$ reduces to the linear fractional Laplacian, up to a dimensional constant $C(N,s)>0$). The reaction is of logistic type, with powers $1<q<r<\p$, where $\p=Np/(N-ps)$ denotes the critical exponent for fractional Sobolev spaces, and $\lambda>0$ is a parameter. Problem $(P_\lambda)$ may is classified in three different types, according to the principal exponent $q>1$:
\begin{itemize}
\item[$(a)$] subdiffusive, if $q<p<r$;
\item[$(b)$] equidiffusive, if $p=q<r$;
\item[$(c)$] superdiffusive, if $p<q<r$.
\end{itemize}
Logistic equations are widely studied mainly because of their important applications in mathematical biology. Indeed, the parabolic semilinear logistic equation describes the evolution and spatial distribution of a biological population in the presence of constant rates of reproduction and mortality (Verhulst's law), see \cite{GM}. This is the obvious reason why, in the study of logistic type equations, authors are usually interested in {\em positive} solutions. More recently, evolutive systems involving logistic terms have been studied as a model for the biological phenomenon of chemotaxis \cite{TW}, and existence of solutions in the presence of a parameter was studied in \cite{AB,CC}. Regarding the elliptic counterpart, it models an equilibrium distribution, see \cite{CDT}. Several existence results for the equidiffusive cases $(b)$, combining variational an topological methods, can be found in \cite{AL,AM,S} (note that multiplicity often includes negative and nodal solutions). Bifurcation results for the superdiffusive case $(c)$ can be found in \cite{IP} for the Dirichlet problem, and in \cite{MP} for the whole space.
\vskip2pt
\noindent
Fractional order equations also have a close connection to mathematical biology, since fractional elliptic operators model space diffusion via L\'evy type random motion with jumps, and hence they can be used to describe the movement of populations, see \cite{BRR,PV}. Studies on logistic equations with several nonlocal operators of fractional order have appeared in recent years, including the square root of the Dirichlet Laplacian \cite{CM}, the spectral Neumann fractional Laplacian \cite{MPV}, and the fractional Laplacian on the whole space \cite{QX}.
\vskip2pt
\noindent
The operator we consider here is both nonlinear and nonlocal. It represents the nonlinear generalization of the fractional Laplacian, and it can be seen as the gradient of the functional $u\mapsto [u]_{s,p}^p/p$ in the fractional Sobolev space $W^{s,p}(\R^N)$ (see Section \ref{sec2} below), as first pointed out in \cite{BCF}. The corresponding eigenvalue problem was studied in \cite{LL}, which led to existence results for general nonlinear reactions via Morse theory in \cite{ILPS}. Due to the nature of the operator, regularity theory for weak solutions requires a considerable effort as most usual techniques (including the Caffarelli-Silvestre extension method) do not apply here. For any $p>1$, H\"older continuity of weak solutions in the interior and up to the boundary was studied in \cite{DKP} and \cite{IMS1}, respectively. In the degenerate case $p>2$, optimal interior H\"older regularity was proved in \cite{BLS}, while a weighted global H\"older regularity result was proved in \cite{IMS2} (the singular case $p\in(1,2)$ is still open).
\vskip2pt
\noindent
The result of \cite{IMS2} is the fractional counterpart  of Lieberman's $C^{1,\alpha}$-regularity result for the classical $p$-Laplacian \cite{L} and yields many applications, such as the equivalence of Sobolev and H\"older local minimizers of the energy functional \cite{IMS3}, the existence of extremal constant sign solutions \cite{FI}, and more recently a Brezis-Oswald type weak comparison principle \cite{IL}. We also recall other interesting related results, such as the study of critical growth and singularity performed in \cite{CMS} and the bifurcation results of \cite{DQ1,PSY}. For further information we refer the reader to the surveys \cite{MS,P}.
\vskip2pt
\noindent
As far as we know, the present literature includes no specific study on the logistic equation for the fractional $p$-Laplacian. This paper aims at filling the gap, by presenting the following general result for existence of solutions to problem $(P_\lambda)$ (in which $\hat\lambda_1>0$ denotes the principal eigenvalue of $\fpl$ in $\Omega$ with Dirichlet conditions):

\begin{theorem}\label{main}
Let $\Omega\subset\R^N$ be a bounded domain with $C^{1,1}$-boundary, $p\ge 2$, $s\in(0,1)$ s.t.\ $ps<N$, and $1<q<r<\p$. Then, the following hold:
\begin{itemize}
\item[$(a)$] if $q<p$, then for all $\lambda>0$ problem $(P_\lambda)$ has a unique solution $u_\lambda>0$, with $u_\lambda>u_\mu$ in $\Omega$ for all $\lambda>\mu>0$ and $u_\lambda\to 0$ as $\lambda\to 0^+$;
\item[$(b)$] if $q=p$, then for all $\lambda\in(0,\hat\lambda_1]$ problem $(P_\lambda)$ has no solution, while for all $\lambda>\hat\lambda_1$ $(P_\lambda)$ has a unique solution $u_\lambda>0$, with $u_\lambda>u_\mu$ in $\Omega$ for all $\lambda>\mu>\hat\lambda_1$ and $u_\lambda\to 0$ as $\lambda\to\hat\lambda_1^+$;
\item[$(c)$] if $q>p$, then there exists $\lambda_*>0$ s.t.\ for all $\lambda\in(0,\lambda_*)$ problem $(P_\lambda)$ has no solution, while $(P_{\lambda_*})$ has at least one solution $u_*>0$, and for all $\lambda>\lambda_*$ $(P_\lambda)$ has at least two solutions $u_\lambda>v_\lambda>0$, with $u_\lambda>u_\mu$ in $\Omega$ for all $\lambda>\mu>\lambda_*$ and $u_\lambda\to u_*$ as $\lambda\to\lambda_*^+$.
\end{itemize}
\end{theorem}

\noindent
More precise statements of the results above can be found in the subsequent Theorems \ref{sub}, \ref{equ}, and \ref{sup}. Our approach is variational, based on critical point theory and comparison-truncation arguments. For the sub- and equidiffusive cases we apply direct minimization and the weak comparison result of \cite{IL} for uniqueness. In the superdiffusive case, we prove a bifurcation result and detect via the mountain pass theorem a second solution for all $\lambda>\lambda_*$.
\vskip2pt
\noindent
We remark that our result is new even in the semiliear case $p=2$ (fractional Laplacian) and in the local case $s=1$ (classical $p$-Laplacian). Bifurcation theorems are proved in \cite{CM} for the superdiffusive logistic equation driven by the square root of the Laplacian, and in \cite{IP} for the classical $p$-Laplacian, but with no information about monotonicity, order between solutions, and convergence. Also, existence and uniqueness for the equidiffusive case with the fractional Laplacian are proved in \cite{QX}.
\vskip2pt
\noindent
A crucial role in our arguments is played by new strong maximum and comparison principles for weak sub- and supersolutions, including a Hopf type property (see Theorems \ref{smp}, \ref{scp}). Previous results of this type were proved in \cite{DQ} and in \cite{J}, respectively, but our versions involve very general reactions and milder restrictions on the constants $p$, $s$ and can be of general interest, since they care applicable to a wide class of problems driven by the fractional $p$-Laplacian.
\vskip4pt
\noindent
{\bf Structure of the paper:} In Section \ref{sec2} we recall some preliminary results and prove new maximum and comparison principles (Subsection \ref{ss1}); in Section \ref{sec3} we deal with the subdiffusive case; in Section \ref{sec4} we deal with the equidiffusive case; and in Section \ref{sec5} we deal with the superdiffusive case.
\vskip4pt
\noindent
{\bf Notation:} For any $a\in\R$, $\nu>0$ we set $a^\nu=|a|^{\nu-1}a$. For any $A\subset\R^N$ we shall set $A^c=\R^N\setminus A$ and denote by $|A|$ the Lebesgue measure of $A$. For any two measurable functions $u,v:\Omega\to\R$, $u\le v$ will mean that $u(x)\le v(x)$ for a.e.\ $x\in\Omega$ (and similar expressions). The positive (resp., negative) part of $u$ is denoted $u^+$ (resp., $u^-$). Every function $u$ defined in $\Omega$ will be identified with its $0$-extension to $\R^N$. If $X$ is an ordered function space, then $X_+$ will denote its non-negative order cone. For all $\nu\in[1,\infty]$, $\|\cdot\|_\nu$ denotes the standard norm of $L^\nu(\Omega)$ (or $L^\nu(\R^N)$, which will be clear from the context). Moreover, $C$ will denote a positive constant whose value may change case by case.

\section{Preliminaries}\label{sec2}

\noindent
Problem $(P_\lambda)$ falls into the following class of Dirichlet problems for the fractional $p$-Laplacian:
\beq\label{dir}
\begin{cases}
\fpl u = f(x,u) & \text{in $\Omega$} \\
u = 0 & \text{in $\Omega^c$,}
\end{cases}
\eeq
where $\Omega$, $p$, $s$ are as in Section \ref{sec1} and the general reaction $f$ satisfies the following hypothesis:
\begin{itemize}[leftmargin=1cm]
\item[${\bf H}$] $f:\Omega\times\R\to\R$ is a Carath\'eodory function s.t.\ for a.e.\ $x\in\Omega$ and all $t\in\R$
\[|f(x,t)| \le c_0(1+|t|^{r-1}) \qquad (c_0>0,\,r\in (p,\p)).\]
\end{itemize}
A variational theory for problem \eqref{dir} was established in the recent literature (see for instance \cite{FI,ILPS,IMS3}). For the reader's convenience, we recall here some of its main features. First, for all measurable $u:\Omega\to\R$ we introduce the Gagliardo seminorm
\[[u]_{s,p, \Omega} = \Big[\iint_{\Omega\times\Omega}\frac{|u(x)-u(y)|^p}{|x-y|^{N+ps}}\,dx\,dy\Big]^\frac{1}{p}\]
setting $[u]_{s, p, \R^{N}}=[u]_{s, p}$. Then, we define the fractional Sobolev spaces
\[W^{s,p}(\Omega) = \big\{u\in L^p(\Omega):\,[u]_{s,p, \Omega}<\infty\big\},\]
\[\w = \big\{u\in W^{s,p}(\R^N):\,u=0 \ \text{in $\Omega^c$}\big\},\]
the latter being a uniformly convex, separable Banach space with norm $\|u\|=[u]_{s,p}$, whose dual space is denoted by $W^{-s,p'}(\Omega)$ (see \cite{DPV}). The embedding $\w\hookrightarrow L^\nu(\Omega)$ is continuous for all $\nu\in[1,p^*_s]$ and compact for all $q\in[1,p^*_s)$. We also recall from \cite[Definition 2.1]{IMS1} the following special space:
\[\widetilde{W}^{s,p}(\Omega) = \Big\{u\in L^p_{\rm loc}(\R^N):\,\exists\,U\Supset\Omega \ \text{s.t. $u\in W^{s,p}(U)$,}\,\int_{\R^N}\frac{|u(x)|^{p-1}}{(1+|x|)^{N+ps}}\, dx<\infty\Big\}.\]
By \cite[Lemma 2.3]{IMS1}, we can define the fractional $p$-Laplacian as a nonlinear operator $\fpl:\widetilde{W}^{s,p}(\Omega)\to W^{-s,p'}(\Omega)$ by setting for all $u,v\in\w$
\[\langle \fpl u,v\rangle = \iint_{\R^N\times\R^N} \frac{(u(x)-u(y))^{p-1}(v(x)-v(y))}{|x-y|^{N+ps}}\,dx\,dy\]
(with the convention $a^{p-1}=|a|^{p-2}a$ established above). Such definition is equivalent to the one given in Section \ref{sec1} as soon as $u$ is smooth enough (for instance, if $u\in\mathcal{S}(\R^N)$).
\vskip2pt
\noindent
Clearly $\w\subset\widetilde{W}^{s,p}(\Omega)$. The restricted operator $\fpl:\w\to W^{-s,p'}(\Omega)$ is continuous, maximal monotone, and enjoys the $(S)_+$-property, namely, whenever $(u_n)$ is a sequence in $\w$ s.t.\ $u_n\rightharpoonup u$ in $\w$ and
\[\limsup_n\langle\fpl u_n,u_n-u\rangle \le 0,\]
then $u_n\to u$ in $\w$ (see \cite[Lemma 2.1]{FI}). From \cite[Lemma 2.1]{IL}, for all $u\in\w$ we have
\beq\label{pnp}
\|u^\pm\|^p \le \langle\fpl u,\pm u^\pm\rangle.
\eeq
Another useful property, referred to as strict $T$-monotonicity, of $\fpl$ is the following (see \cite[proof of Lemma 3.2]{FI}):

\begin{proposition}\label{stm}
Let $u,v\in\widetilde{W}^{s,p}(\Omega)$ s.t.\ $(u-v)^+\in\w$ satisfy
\[\langle\fpl u-\fpl v,(u-v)^+\rangle \le 0.\]
Then, $u\le v$ in $\Omega$.
\end{proposition}

\noindent
We say that $u\in\widetilde{W}^{s,p}(\Omega)$ is a (weak) supersolution of \eqref{dir} if $u\ge 0$ in $\Omega^c$ and for all $v\in\w_+$
\[\langle\fpl u,v\rangle \ge \int_\Omega f(x,u)v\,dx,\]
and similarly we define a (weak) subsolution. Finally, $u\in\w$ is a (weak) solution of \eqref{dir} if it is both a super- and a subsolution, i.e., if for all $v\in\w$
\[\langle\fpl u,v\rangle = \int_\Omega f(x,u)v\,dx.\]
In such cases we write that weakly in $\Omega$
\[\fpl u = \ (\ge,\,\le) \ f(x,u).\]
From \cite[Theorem 3.3]{CMS} we have the following a priori bound on the solutions:

\begin{proposition}\label{apb}
Let ${\bf H}$ hold, $u\in\w$ be a solution of \eqref{dir}. Then, $u\in L^\infty(\Omega)$ with $\|u\|_\infty\le C(\|u\|)$.
\end{proposition}

\noindent
Classical nonlinear regularity theory does not apply to fractional order equations, whose solutions fail to be $C^1$ in general. Nevertheless, weighted H\"older continuity can replace higher smoothness in most cases. We set ${\rm d}_\Omega(x)={\rm dist}(x,\Omega^c)$ for all $x\in\R^N$ and define the following space:
\[\cs = \Big\{u \in C^0(\overline{\Omega}): \frac{u}{\ds} \ \text{has a continuous extension to} \ \overline{\Omega} \Big\},\]
a Banach space under the norm $\|u\|_{0,s}=\|u/\ds\|_\infty$. By \cite[Lemma 5.1]{ILPS}, the positive order cone $\cs_+$ has a nonempty interior
\[{\rm int}(\cs_+) = \Big\{u\in\cs:\,\inf_\Omega\frac{u}{\ds}>0\Big\}.\]
Similarly, for any $\alpha\in(0,1)$ we set
\[C_s^{\alpha}(\overline{\Omega})= \Big\{u \in C^0(\overline{\Omega}): \frac{u}{\ds} \ \text{has a $\alpha$-H\"older continuous extension to} \ \overline{\Omega} \Big\},\]
a Banach space under the norm
\[\|u\|_{\alpha,s}= \|u\|_{0,s} + \sup_{x \neq y} \frac{|u(x)/\ds(x) - u(y)/\ds(y)|}{|x-y|^{\alpha}}.\]
By the Ascoli-Arzel\`a theorem, $C_s^{\alpha}(\overline{\Omega})\hookrightarrow\cs$ with compact embedding for all $\alpha\in(0,1)$. From Proposition \ref{apb} and \cite[Theorem 1.1]{IMS2} we have the following weighted H\"older regularity result:

\begin{proposition}\label{reg}
Let ${\bf H}$ hold, $u\in\w$ be a solution of \eqref{dir}. Then there exists $\alpha\in(0,s]$, independent of $u$, s.t.\ $u\in C^\alpha_s(\overline\Omega)$ and $\|u\|_{\alpha,s}\le C(\|u\|)$.
\end{proposition}

\noindent
From \cite[Proposition 2.8]{IL} we have the following weak comparison principle under a special monotonicity assumption of Brezis-Oswald type:

\begin{proposition}\label{wcp}
Let ${\bf H}$ hold and assume that
\[t\mapsto\frac{f(x,t)}{t^{p-1}}\]
is decreasing in $(0,\infty)$ for a.e.\ $x\in\Omega$. Let $u,v\in{\rm int}(\cs_+)\cap\w$ be a subsolution and a supersolution, respectively, of \eqref{dir}. Then, $u\le v$ in $\Omega$.
\end{proposition}

\noindent
The energy functional for problem \eqref{dir} is defined by setting for all $u\in\w$
\[\Phi(u) = \frac{\|u\|^p}{p}-\int_\Omega F(x,u)\,dx,\]
where we have set for all $(x,t)\in\Omega\times\R$
\[F(x,t) = \int_0^t f(x,\tau)\,dx.\]
By classical results, we have $\Phi\in C^1(\w)$, and $u\in\w$ is a solution of \eqref{dir} iff $\Phi'(u)=0$ in $W^{-s,p'}(\Omega)$. Besides, by \cite[Proposition 2.1]{ILPS} $\Phi$ satisfies a bounded $(PS)$-condition, namely, whenever $(u_n)$ is a bounded sequence in $\w$ s.t.\ $(\Phi(u_n))$ is bounded in $\R$ and $\Phi'(u_n)\to 0$ in $W^{-s,p'}(\Omega)$, then $(u_n)$ has a convergent subsequence. In this connection, we recall from \cite[Theorem 1.1]{IMS3} the following equivalence principle for Sobolev and H\"older local minimizers of $\Phi$:

\begin{proposition}\label{svh}
Let ${\bf H}$ hold, $u\in\w$. Then, the following are equivalent:
\begin{enumroman}
\item\label{svh1} there exists $\rho>0$ s.t.\ $\Phi(u+v)\ge\Phi(u)$ for all $v\in\w\cap C_s^0(\overline\Omega)$, $\|v\|_{0,s}\le\rho$;
\item\label{svh2} there exists $\sigma>0$ s.t.\ $\Phi(u+v)\ge\Phi(u)$ for all $v\in\w$, $\|v\|\le\sigma$.
\end{enumroman}
\end{proposition}

\noindent
Regarding the spectral properties of the fractional $p$-Laplacian, we refer the reader to \cite{LL}. We just recall that the eigenvalue problem is stated as
\beq\label{evp}
\begin{cases}
\fpl u = \lambda u^{p-1} & \text{in $\Omega$} \\
u = 0 & \text{in $\Omega$.}
\end{cases}
\eeq
The principal eigenvalue of \eqref{evp} is defined as
\beq\label{pev}
\hat\lambda_1 = \inf_{u\in\w\setminus\{0\}}\frac{\|u\|^p}{\|u\|_p^p},
\eeq
and it is simple and isolated with a unique positive eigenfunction $\hat u_1\in{\rm int}(\cs_+)$ s.t.\ $\|u\|_p=1$.

\subsection{Strong maximum and comparison principles}\label{ss1}

As mentioned in Section \ref{sec1}, a strong maximum principle and a Hopf type lemma for the fractional $p$-Laplacian were proved in \cite[Theorems 1.2, 1.5]{DQ}, while a strong comparison principle was obtained in \cite[Theorem 1.1]{J}. Nevertheless, the strong comparison principle of \cite{J} does not fit with our purposes for two reasons: first, in the degenerate case $p>2$ it requires some special relations between the parameters $p$ and $s$ which, combined with the optimal H\"older continuity proved in \cite{BLS}, lead to the quite restrictive condition $s\le 1/p'$; second, the result only ensures that the difference between the super- and the subsolution is positive in $\Omega$, while we need to prove that such difference lies in ${\rm int}(\cs_+)$.
\vskip2pt
\noindent
Motivated by such difficulties, we present here a new pair of results, following an alternative approach based on the nonlocal superposition principle introduced in \cite{IMS2}. We begin with a strong maximum principle (including a Hopf type boundary property):

\begin{theorem}\label{smp}
Let $g\in C^0(\R)\cap BV_{\rm loc}(\R)$, $u\in\w\cap C^0(\overline\Omega)$, $u\not\equiv 0$ s.t.\
\[\begin{cases}
\fpl u+g(u) \ge g(0) & \text{weakly in $\Omega$} \\
u \ge 0 & \text{in $\Omega$.}
\end{cases}\]
Then,
\[\inf_\Omega\frac{u}{\ds} > 0.\]
In particular, if $u\in\cs$, then $u\in{\rm int}(\cs_+)$.
\end{theorem}
\begin{proof}
By Jordan's decomposition, we can find $g_1,g_2\in C^0(\R)$ nondecreasing s.t.\ $g(t)=g_1(t)-g_2(t)$ for all $t\in\R$, and $g_1(0)=0$. So we have weakly in $\Omega$
\begin{align*}
\fpl u +g_1(u) &= \fpl u+g(u)+g_2(u) \\
&\ge g(0)+g_2(0) = 0.
\end{align*}
Thus, without loss of generality we may assume that $g$ is nondecreasing and $g(0)=0$. In order to prove our assertion, we need a lower barrier for $u$. Let us consider the following torsion problem:
\beq\label{smp1}
\begin{cases}
\fpl v = 1 & \text{in $\Omega$} \\
v = 0 & \text{in $\Omega^c$.}
\end{cases}
\eeq
By convexity, \eqref{smp1} has a unique solution $v\in\w$, which by \cite[Lemma 2.3]{IMS2} satisfies $v\ge c\ds$ in $\Omega$, for some $c>0$. By Proposition \ref{reg} we have $v\in C^\alpha_s(\Omega)$, in particular $v$ is continuous. So, since $u\not\equiv 0$, we can find $x_0\in\Omega$, $\rho,\eps>0$, and $\eta_0\in(0,1)$ s.t.\ $\overline{B}_\rho(x_0)\subset\Omega$ and
\beq\label{smp2}
\sup_{\overline{B}_\rho(x_0)}\eta_0 v < \inf_{\overline{B}_\rho(x_0)}u-\eps.
\eeq
Set for all $x\in\R^N$, $\eta\in(0,\eta_0]$
\[w_\eta(x) = \begin{cases}
\eta v(x) & \text{if $x\in\overline{B}_{\rho/2}^c(x_0)$} \\
u(x) & \text{if $x\in \overline{B}_{\rho/2}(x_0)$.}
\end{cases}\]
First, by \eqref{smp2} we have $w_\eta\le u$ in $\overline{B}_\rho(x_0)$. Besides, $w_\eta\in\widetilde{W}^{s,p}(\Omega\setminus\overline{B}_\rho(x_0))$ and by the nonlocal superposition principle \cite[Proposition 2.6]{IMS2} we have weakly in   $\Omega\setminus\overline{B}_\rho(x_0)$ 
\begin{align*}
\fpl w_\eta(x) &= \fpl(\eta v)(x)+2\int_{\overline{B}_{\rho/2}(x_0)}\frac{(\eta v(x)-u(y))^{p-1}-(\eta v(x)-\eta v(y))^{p-1}}{|x-y|^{N+ps}}\,dy \\
&\le \eta^{p-1}\fpl v(x)+C\int_{\overline{B}_{\rho/2}(x_0)}\frac{(\eta v(y)-u(y))^{p-1}}{|x-y|^{N+ps}}\,dy \\
&\le \eta^{p-1}+C\int_{\overline{B}_{\rho/2}(x_0)}\frac{-\eps^{p-1}}{(\rho/2)^{N+ps}}\,dy \\
&\le \eta^{p-1}-C_\rho\eps^{p-1}
\end{align*}
for some $C_\rho>0$ independent of $\eps$, where we have used \eqref{smp1}, \eqref{smp2}, and the following elementary inequality:
\[(a-b)^{p-1}-(a-c)^{p-1} \le 2^{2-p}(c-b)^{p-1},\]
holding for all $a,b,c\in\R$, $b\ge c$ (see \cite[eq. (2.7)]{IMS1}). So, for all $\eta\in(0,\eta_0]$ small enough we have weakly in $\Omega\setminus\overline{B}_\rho(x_0)$
\[\fpl w_\eta < -C\eps^{p-1}.\]
Since $g(w_\eta)\to 0$ uniformly in $\Omega\setminus\overline{B}_\rho(x_0)$ as $\eta\to 0^+$, for an even smaller $\eta\in(0,\eta_0]$ we have (in a weak sense)
\[\begin{cases}
\fpl w_\eta+g(w_\eta) \le 0 \le \fpl u+g(u) & \text{in $\Omega\setminus\overline{B}_\rho(x_0)$} \\
w_\eta \le u & \text{in $(\Omega\setminus\overline{B}_\rho(x_0))^c$.}
\end{cases}\]
Since $w_\eta\le u$ in $\overline{B}_\rho(x_0)$, we have $(w_\eta-u)^+\in W^{s,p}_{0}(\Omega\setminus \overline{B}_\rho(x_0))$. So we can employ such function to test the inequality above and get
\[\langle\fpl w_\eta-\fpl u,(w_\eta-u)^+\rangle \le \int_{\Omega\setminus\overline{B}_\rho(x_0)}(g(u)-g(w_\eta))(w_\eta-u)^+\,dx,\]
and the latter is negative by monotonicity of $g$. By Proposition \ref{stm} we have $w_\eta\le u$ in $\Omega\setminus\overline{B}_\rho(x_0)$. Combining with \eqref{smp2} we get in $\Omega$
\[u \ge \eta v \ge \eta c\ds,\]
hence the conclusion. In particular, if $u\in\cs$, then clearly we have $u\in{\rm int}(\cs_+)$.
\end{proof}

\noindent
With a similar technique we prove a strong comparison principle for general autonomous reactions:

\begin{theorem}\label{scp}
Let $g\in C^0(\R)\cap BV_{\rm loc}(\R)$, $u,v\in\w\cap C^0(\overline\Omega)$ s.t.\ $u\not\equiv v$, $K>0$ satisfy
\[\begin{cases}
\fpl v+g(v) \le \fpl u+g(u) \le K & \text{weakly in $\Omega$} \\
0 < v \le u & \text{in $\Omega$.}
\end{cases}\]
Then, $u>v$ in $\Omega$. In particular, if $u,v\in{\rm int}(\cs_+)$, then $u-v\in{\rm int}(\cs_+)$.
\end{theorem}
\begin{proof}
As in Theorem \ref{smp}, we may assume $g$ nondecreasing. By continuity, we can find $x_0\in\Omega$, $\rho,\eps>0$ s.t.\ $\overline{B}_\rho(x_0)\subset\Omega$ and
\[\sup_{\overline{B}_\rho(x_0)}v < \inf_{\overline{B}_\rho(x_0)}u-\eps.\]
Hence, for all $\eta>1$ close enough to $1$ we have
\beq\label{scp1}
\sup_{\overline{B}_\rho(x_0)}\eta v < \inf_{\overline{B}_\rho(x_0)}u-\frac{\eps}{2}.
\eeq
Define $w_\eta\in\widetilde{W}^{s,p}(\Omega\setminus\overline{B}_\rho(x_0))$ as in Theorem \ref{smp}, so by \eqref{scp1} we have $w_\eta\le u$ in $\overline{B}_\rho(x_0)$. Applying nonlocal superposition as in the previous proof we have weakly in $\Omega\setminus\overline{B}_\rho(x_0)$
\[\fpl w_\eta \le \eta^{p-1}\fpl v-C_\rho\eps^{p-1},\]
for some $C_\rho>0$ and all $\eta>1$ close enough to $1$. Further, we have weakly in $\Omega\setminus\overline{B}_\rho(x_0)$
\begin{align*}
\fpl w_\eta+g(w_\eta) &\le \eta^{p-1}\fpl v+g(w_\eta)-C_\rho\eps^{p-1} \\
&\le \eta^{p-1}\big(\fpl v+g(v)\big)+\big(g(w_\eta)-\eta^{p-1}g(v)\big)-C_\rho\eps^{p-1} \\
&\le \eta^{p-1}\big(\fpl u+g(u)\big)+\big(g(w_\eta)-\eta^{p-1}g(v)\big)-C_\rho\eps^{p-1} \\
&\le \fpl u+g(u)+K(\eta^{p-1}-1)+\big(g(w_\eta)-\eta^{p-1}g(v)\big)-C_\rho\eps^{p-1}.
\end{align*}
Since
\[K(\eta^{p-1}-1)+\big(g(w_\eta)-\eta^{p-1}g(v)\big) \to 0\]
uniformly in $\Omega\setminus\overline{B}_\rho(x_0)$ as $\eta\to 1^+$, we have for all $\eta>1$ close enough to $1$
\[\begin{cases}
\fpl w_\eta+g(w_\eta) \le \fpl u+g(u) & \text{weakly in $\Omega\setminus\overline{B}_\rho(x_0)$} \\
w_\eta \le u & \text{in $(\Omega\setminus\overline{B}_\rho(x_0))^c$.}
\end{cases}\]
Testing with $(w_\eta-u)^+\in W^{s, p}_0(\Omega\setminus\overline{B}_\rho(x_0))$, recalling the monotonicity of $g$, and applying Proposition \ref{stm} we get $u\ge w_\eta$ in $\Omega\setminus\overline{B}_\rho(x_0)$. So we have in $\Omega$
\[u \ge \eta v >v,\]
hence the conclusion. In particular, if $u,v\in{\rm int}(\cs_+)$, then clearly
\[\inf_\Omega\frac{u-v}{\ds} \ge \inf_\Omega\frac{(\eta-1)v}{\ds} > 0,\]
so $u-v\in{\rm int}(\cs_+)$.
\end{proof}

\begin{remark}\label{loc}
Both results above exhibit unexpected differences when compared to the corresponding local versions, i.e., the case of the classical $p$-Laplacian. For example, according to Theorem \ref{smp}, the strong maximum principle holds for supersolutions of the Dirichlet problem
\[\begin{cases}
\fpl u+u^\sigma = 0 & \text{in $\Omega$} \\
u = 0 & \text{in $\Omega^c$}
\end{cases}\]
for any $\sigma>0$, while for $s=1$ the same is not true when $\sigma<p-1$ due to the possible presence of dead cores (see \cite[p.\ 204]{PS1}). Also, the strong comparison principle of Theorem \ref{scp} includes cases which are excluded in the local case (see \cite[Example 4.1]{CT}). This is essentially due to the nonlocal nature of the operator.
\end{remark}

 \section{The subdiffusive case}\label{sec3}

\noindent
In this section we study problem $(P_\lambda)$ with $1<q<p<r<\p$. For all $\lambda>0$, $t\in\R$ we set
\[f_\lambda(t) = \lambda(t^+)^{q-1}-(t^+)^{r-1},\]
\[F_\lambda(t) = \int_0^tf_\lambda(\tau)\,d\tau = \lambda\frac{(t^+)^q}{q}-\frac{(t^+)^r}{r}.\]
Note that $f_\lambda:\R\to\R$ satisfies hypotheses ${\bf H}$ as stated in Section \ref{sec2}. So we may set for all $u\in\w$
\[\Phi_\lambda(u) = \frac{\|u\|^p}{p}-\int_\Omega F_\lambda(u)\,dx,\]
and deduce that $\Phi_\lambda\in C^1(\w)$ and the positive critical points of $\Phi_\lambda$ coincide with the solutions of $(P_\lambda)$.
\vskip2pt
\noindent
In this case we have the following global existence and uniqueness result (corresponding to case $(a)$ of Theorem \ref{main}):

\begin{theorem}\label{sub}
Let $1<q<p<r<\p$. Then, for all $\lambda>0$ problem $(P_\lambda)$ has a unique solution $u_\lambda\in{\rm int}(\cs_+)$, s.t.\ $u_\lambda-u_\mu\in{\rm int}(\cs_+)$ for all $\lambda>\mu>0$ and $u_\lambda\to 0$ in both $\w$ and $\cs$ as $\lambda\to 0^+$.
\end{theorem}
\begin{proof}
Fix any $\lambda>0$. We will find the solution of $(P_\lambda)$ by direct minimization. First we prove that $\Phi_\lambda$ is coercive. Indeed, since $q<r$, the mapping $F_\lambda$ is clearly bounded from above, i.e., there exists $C>0$ s.t.\ $F_\lambda(t) \le C$ for all $t\in\R$. So, for all $u\in\w$ we have
\[\Phi_\lambda(u) \ge \frac{\|u\|^p}{p}-C|\Omega|,\]
and the latter tends to $\infty$ as $\|u\|\to\infty$. Besides, by the compact embeddings $\w\hookrightarrow L^q(\Omega),\,L^r(\Omega)$, it is easily seen that $\Phi_\lambda$ is sequentially weakly lower semicontinuous in $\w$. So, there exists $u_\lambda\in\w$ s.t.\
\beq\label{sub1}
\Phi_\lambda(u_\lambda) = \inf_{\w}\Phi_\lambda =: m_\lambda.
\eeq
Besides, let $\hat u_1\in{\rm int}(\cs_+)$ be defined as in Section \ref{sec2}, then for all $\tau>0$
\[\Phi_\lambda(\tau\hat u_1) = \tau^p\frac{\|\hat u_1\|^p}{p}-\lambda\tau^q\frac{\|\hat u_1\|_q^q}{q}+\tau^r\frac{\|\hat u_1\|_r^r}{r},\]
and the latter is negative for all $\tau>0$ small enough (recall that $q<p<r$). So, in \eqref{sub1} we have $m_\lambda<0$, implying $u_\lambda\not \equiv 0$. From \eqref{sub1} we deduce that $\Phi'_\lambda(u_\lambda)=0$ in $W^{-s,p'}(\Omega)$, i.e., we have weakly in $\Omega$
\beq\label{sub2}
\fpl u_\lambda = f_\lambda(u_\lambda).
\eeq
By Lemma \ref{reg} we have $u_\lambda\in C^\alpha_s(\overline\Omega)$. Besides, testing \eqref{sub2} with $-u_\lambda^-\in\w$ and applying \eqref{pnp}, we have
\[\|u_\lambda^-\|^p \le \langle\fpl u_\lambda,-u_\lambda^-\rangle = \int_\Omega f_\lambda(u_\lambda)(-u_\lambda^-)\,dx = 0.\]
so $u_\lambda\ge 0$. Now Theorem \ref{smp} implies $u_\lambda\in{\rm int}(\cs_+)$, so $u_\lambda$ solves $(P_\lambda)$.
\vskip2pt
\noindent
Next we prove uniqueness. Let $v_\lambda\in{\rm int}(\cs_+)$ be another solution of $(P_\lambda)$. We have for all $t>0$
\[\frac{f_\lambda(t)}{t^{p-1}} = \lambda t^{q-p}-t^{r-p},\]
and such mapping is decreasing in $(0,\infty)$. Applying Proposition \ref{wcp} twice, we have $u_\lambda=v_\lambda$.
\vskip2pt
\noindent
To see monotonicity, let $0<\mu<\lambda$, and $u_\mu,u_\lambda\in{\rm int}(\cs_+)$ be the solutions of $(P_\mu)$, $(P_\lambda)$, respectively. We have weakly in $\Omega$
\[\fpl u_\mu < \lambda u_\mu^{q-1}-u_\mu^{r-1},\]
so $u_\mu$ is a strict subsolution of $(P_\lambda)$. By Proposition \ref{wcp} again we have $u_\mu\le u_\lambda$ in $\Omega$. This in turn implies that weakly in $\Omega$
\[\fpl u_\mu+u_\mu^{r-1} = \mu u_\mu^{q-1} < \lambda u_\lambda^{q-1} = \fpl u_\lambda+u_\lambda^{r-1}.\]
Since $g(t)=t^{r-1}$ is continuous and with locally bounded variation, we can apply Theorem \ref{scp} and see that $u_\lambda-u_\mu\in{\rm int}(\cs_+)$.
\vskip2pt
\noindent
Finally, let $(\lambda_n)$ be a decreasing sequence in $(0,\infty)$ s.t.\ $\lambda_n\to 0^+$, and $u_n\in{\rm int}(\cs_+)$ be the solution of $(P_{\lambda_n})$ for all $n\in\N$, i.e., we have weakly in $\Omega$
\beq\label{sub3}
\fpl u_n = f_{\lambda_n}(u_n).
\eeq
Since $q<p$ and $(\lambda_n)$ is decreasing, we can find $C>0$ s.t.\ for all $n\in\N$, $t\in\R$
\[f_{\lambda_n}(t)t \le C.\]
Testing \eqref{sub3} with $u_n\in\w$, for all $n\in\N$ we have
\[\|u_n\|^p = \langle\fpl u_n,u_n\rangle = \int_\Omega f_{\lambda_n}(u_n)u_n\,dx \le C|\Omega|.\]
So, $(u_n)$ is a bounded sequence in $\w$. By reflexivity and the compact embeddings $\w\hookrightarrow L^q(\Omega),\,L^r(\Omega)$, we can pass to a subsequence s.t.\ $u_n\rightharpoonup u_0$ in $\w$ and $u_n\to u_0$ in both $L^q(\Omega)$ and $L^r(\Omega)$. Testing \eqref{sub3} with $(u_n-u_0)\in\w$ and using H\"older's inequality, we have for all $n\in\N$
\begin{align*}
\langle\fpl u_n,u_n-u_0\rangle &= \int_\Omega(\lambda_n u_n^{q-1}-u_n^{r-1})(u_n-u_0)\,dx \\
&\le \lambda_1\|u_n\|_q^{q-1}\|u_n-u_0\|_q+\|u_n\|_r^{r-1}\|u_n-u_0\|_r,
\end{align*}
and the latter tends to $0$ as $n\to \infty$. By the $(S)_+$-property of $\fpl$, we have $u_n\to u_0$ in $\w$. So we can pass to the limit in \eqref{sub3} as $n\to\infty$ and get weakly in $\Omega$
\[\fpl u_0 = -u_0^{r-1}.\]
Testing with $u_0\in\w$ we have
\[\|u_0\|^p+\|u_0\|_r^r = 0,\]
i.e., $u_0=0$. Plus, we note that, by \eqref{sub3} and Proposition \ref{reg}, $(u_n)$ is bounded in $C^\alpha_s(\overline\Omega)$, hence, passing to a further subsequence, $u_n\to 0$ in $\cs$. Recalling that $\lambda\mapsto u_\lambda$ is strictly increasing, we conclude that globally $u_\lambda\to 0$ in both $\w$ and $\cs$, as $\lambda\to 0^+$.
\end{proof}

\section{The equidiffusive case}\label{sec4}

\noindent
In this short section we assume $2\le q=p<r<\p$, a case that does not differ too much from the previous one, except that the threshold for the parameter $\lambda$ turns out to be the principal eigenvalue $\hat\lambda_1>0$ defined in \eqref{pev}.  We define $f_\lambda$, $F_\lambda$, and $\Phi_\lambda$ as in Section \ref{sec3}.
\vskip2pt
\noindent
Our existence and uniqueness result (corresponding to case $(b)$ of Theorem \ref{main}) is the following:

\begin{theorem}\label{equ}
Let $2\le q=p<r<\p$. Then, for all $\lambda\in(0,\hat\lambda_1]$ problem $(P_\lambda)$ has no solution, while for all $\lambda>\hat\lambda_1$ problem $(P_\lambda)$ has a unique solution $u_\lambda\in{\rm int}(\cs_+)$, s.t.\ $u_\lambda-u_\mu\in{\rm int}(\cs_+)$ for all $\lambda>\mu>\hat\lambda_1$ and $u_\lambda\to 0$ in both $\w$ and $\cs$ as $\lambda\to\hat\lambda_1^+$.
\end{theorem}
\begin{proof}
First, fix $\lambda\in(0,\hat\lambda_1]$. Assume that $u\in\w_+$ satisfies weakly in $\Omega$
\beq\label{equ1}
\fpl u = \lambda u^{p-1}-u^{r-1}.
\eeq
Testing \eqref{equ1} with $u\in\w$ and applying \eqref{pev}, we have
\[0 = \|u\|^p-\lambda\|u\|_p^p+\|u\|_r^r \ge (\hat\lambda_1-\lambda)\|u\|_p^p+\|u\|_r^r \ge \|u\|_r^r,\]
hence $u=0$. So $(P_\lambda)$ admits no solution.
\vskip2pt
\noindent
Now let $\lambda>\hat\lambda_1$. Arguing as in Theorem \ref{sub}, we see that $\Phi_\lambda$ has a global minimizer $u_\lambda\in\w_+$. Besides, for all $\tau>0$ we have
\begin{align*}
\Phi_\lambda(\tau\hat u_1) &= \tau^p\Big[\frac{\|\hat u_1\|^p}{p}-\lambda\frac{\|\hat u_1\|_p^p}{p}\Big]+\tau^r\frac{\|\hat u_1\|_r^r}{r} \\
&= \tau^p\frac{\hat\lambda_1-\lambda}{p}+\tau^r\frac{\|\hat u_1\|_r^r}{r},
\end{align*}
and the latter is negative for $\tau>0$ small enough (as $p<r$). So, $u_\lambda\not\equiv 0$. The rest of the proof follows exactly as in Theorem \ref{sub}.
\end{proof}

\section{The superdiffusive case}\label{sec5}

\noindent
In this final section we assume $2\le p<q<r<\p$ and define $f_\lambda$, $F_\lambda$, $\Phi_\lambda$ as in Section \ref{sec3}. Let
\[\Lambda = \big\{\lambda>0:\,\text{$(P_\lambda)$ has a solution $u_\lambda\in{\rm int}(\cs_+)$}\big\}.\]
In the following lemmas we shall investigate the structure of the set $\Lambda$ and additional properties of solutions. We begin with a lower bound:

\begin{lemma}\label{exi}
We have $\Lambda\neq\emptyset$ and $\lambda_*:=\inf\Lambda>0$.
\end{lemma}
\begin{proof}
Fix $\lambda>0$. As in the proof of Theorem \ref{sub} we find $u_\lambda\in\w_+$ s.t.\
\beq\label{exi1}
\Phi_\lambda(u_\lambda) = \inf_{\w}\Phi_\lambda =: m_\lambda.
\eeq
Let $\hat u_1\in{\rm int}(\cs_+)$ be defined as in Section \ref{sec2}, then we have
\[\Phi_\lambda(\hat u_1) = \frac{\|\hat u_1\|^p}{p}-\lambda\frac{\|\hat u_1\|_q^q}{q}+\frac{\|\hat u_1\|_r^r}{r},\]
which tends to $-\infty$ as $\lambda\to\infty$. So, for all $\lambda>0$ big enough we have $m_\lambda<0$ in \eqref{exi1}, hence $u_\lambda\neq 0$. As in Theorem \ref{sub} we see that $u_\lambda\in{\rm int}(\cs_+)$ and it solves $(P_\lambda)$. Thus $\lambda\in\Lambda$.
\vskip2pt
\noindent
We claim that there exists $\lambda_0>0$ s.t.\ for all $t\ge 0$
\beq\label{exi2}
f_{\lambda_0}(t) \le \hat\lambda_1 t^{p-1},
\eeq
with $\hat\lambda_1>0$ defined by \eqref{pev}. Indeed, since $p<q<r$ we have for any $\lambda>0$
\[\lim_{t\to 0^+}\frac{f_\lambda(t)}{t^{p-1}} = 0, \ \lim_{t\to\infty}\frac{f_\lambda(t)}{t^{p-1}} = -\infty.\]
So we can find $\delta\in(0,1)$ s.t.\ for all $t\in(0,\delta)\cup(\delta^{-1},\infty)$ and all $\lambda\in(0,1]$
\[f_\lambda(t) \le \hat\lambda_1 t^{p-1}.\]
Now set
\[\lambda_0 = \min\{\hat\lambda_1\delta^{q-p},1\} > 0.\]
Then, for all $t\in[\delta,\delta^{-1}]$ we have
\[f_{\lambda_0}(t) < \lambda_0 t^{q-1} \le \hat\lambda_1 t^{p-1},\]
hence \eqref{exi2} holds for all $t\ge 0$. We prove that $\inf\Lambda\ge\lambda_0$, arguing by contradiction. Assume that for some $\lambda\in(0,\lambda_0)$ problem $(P_\lambda)$ has a solution $u_\lambda\in{\rm int}(\cs_+)$. Testing with $u_\lambda\in\w$ and using \eqref{exi2} we get
\[\|u_\lambda\|^p = \int_\Omega f_\lambda(u_\lambda)u_\lambda\,dx < \int_\Omega f_{\lambda_0}(u_\lambda)u_\lambda\,dx \le \hat\lambda_1\|u_\lambda\|_p^p,\]
against \eqref{pev}.
\end{proof}

\noindent
Next we prove that $\Lambda$ is a half-line and the mapping $\lambda\mapsto u_\lambda$ is strictly increasing:

\begin{lemma}\label{mon}
If $\lambda>\lambda_*$ then $\lambda\in\Lambda$, besides $u_\lambda-u_\mu\in{\rm int}(\cs_+)$ for all $\lambda>\mu>\lambda_*$.
\end{lemma}
\begin{proof}
Fix $\lambda>\lambda_*$. Then we can find $\mu\in\Lambda$ s.t.\ $\mu<\lambda$, and a solution $u_\mu\in{\rm int}(\cs_+)$ of $(P_\mu)$. We have weakly in $\Omega$
\beq\label{mon1}
\fpl u_\mu = f_\mu(u_\mu) < f_\lambda(u_\mu),
\eeq
i.e., $u_\mu$ is a strict subsolution of $(P_\lambda)$. We use $u_\mu$ to truncate the reaction $f_\lambda$. Set for all $(x,t)\in\Omega\times\R$
\[\hat f_\lambda(x,t) = \begin{cases}
f_\lambda(u_\mu(x)) & \text{if $t\le u_\mu(x)$} \\
f_\lambda(t) & \text{if $t>u_\mu(x)$}
\end{cases}\]
and
\[\hat F_\lambda(x,t) = \int_0^t \hat f_\lambda(x,\tau)\,d\tau.\]
So $\hat f_\lambda:\Omega\times\R\to\R$ satisfies ${\bf H}$. Set for all $u\in\w$
\[\hat\Phi_\lambda(u) = \frac{\|u\|^p}{p}-\int_\Omega \hat F_\lambda(x,u)\,dx,\]
then as in Section \ref{sec2} it is seen that $\hat\Phi_\lambda\in C^1(\w)$. Reasoning as in Theorem \ref{sub} we also see that $\hat\Phi_\lambda$ is coercive and sequentially weakly l.s.c., so there exists $u_\lambda\in\w$ s.t.\
\[\hat\Phi_\lambda(u_\lambda) = \inf_{\w}\hat\Phi_\lambda.\]
As a consequence we have $\hat\Phi_\lambda'(u_\lambda)=0$ in $W^{-s,p'}(\Omega)$, i.e., weakly in $\Omega$
\beq\label{mon2}
\fpl u_\lambda = \hat f_\lambda(x,u).
\eeq
Testing \eqref{mon2} with $(u_\mu-u_\lambda)^+\in\w_+$ we get
\begin{align*}
\langle\fpl u_\lambda,(u_\mu-u_\lambda)^+\rangle &= \int_\Omega \hat f_\lambda(x,u_\lambda)(u_\mu-u_\lambda)^+\,dx \\
&= \int_\Omega f_\lambda(u_\mu)(u_\mu-u_\lambda)^+\,dx,
\end{align*}
which along with \eqref{mon1} gives
\[\langle\fpl u_\mu-\fpl u_\lambda,(u_\mu-u_\lambda)^+\rangle \le 0.\]
By Proposition \ref{stm} we have $u_\mu\le u_\lambda$ in $\Omega$. So \eqref{mon2} rephrases as
\[\fpl u_\lambda = f_\lambda(u_\lambda)\]
weakly in $\Omega$, and moreover $u_\lambda>0$ in $\Omega$. As in Lemma \ref{exi} we see that $u_\lambda\in{\rm int}(\cs_+)$ and it solves $(P_\lambda)$, so $\lambda\in\Lambda$.
\vskip2pt
\noindent
Finally, for all $\lambda>\mu>\lambda_*$ we have weakly in $\Omega$
\[\fpl u_\mu+u_\mu^{r-1} = \mu u_\mu^{q-1} < \lambda u_\lambda^{q-1} = \fpl u_\lambda+u_\lambda^{r-1},\]
as well as $0<u_\mu\le u_\lambda$ in $\Omega$. By Theorem \ref{scp} we conclude that $u_\lambda-u_\mu\in{\rm int}(\cs_+)$.
\end{proof}

\noindent
Note that in Lemma \ref{mon} we cannot use Proposition \ref{wcp} to prove the monotonicity of $\lambda\mapsto u_\lambda$, as we did in sub- and equidiffusive cases: this is due to the fact that $t\mapsto f_\lambda(t)/t^{p-1}$ is not a decreasing mapping in $(0,\infty)$ (recall that $q>p$). The same reason prevents the use of Proposition \ref{wcp} to prove uniqueness of the solution.
\vskip2pt
\noindent
In fact, for $\lambda>\lambda_*$ we detect at least one more solution beside $u_\lambda$:

\begin{lemma}\label{two}
For all $\lambda>\lambda_*$ there exists a second solution $v_\lambda\in{\rm int}(\cs_+)$ of $(P_\lambda)$ s.t.\ $u_\lambda-v_\lambda\in{\rm int}(\cs_+)$.
\end{lemma}
\begin{proof}
Fix $\lambda>\lambda_*$. As in Lemma \ref{mon} we pick $\mu\in\Lambda$ s.t.\ $\lambda_*<\mu<\lambda$, define $\hat\Phi_\lambda\in C^1(\w)$, and find a global minimizer $u_\lambda\in{\rm int}(\cs_+)$, which solves $(P_\lambda)$ and satisfies $u_\lambda-u_\mu\in{\rm int}(\cs_+)$. Set now
\[V=\big\{u_\mu+v:\,v\in{\rm int}(\cs_+)\big\},\]
an open set in $\cs$ containing $u_\lambda$. For all $x\in\Omega$, $t>u_\mu(x)$ we have
\begin{align*}
\hat F_\lambda(x,t) &= \int_0^{u_\mu(x)}f_\lambda(u_\mu(x))\,d\tau+\int_{u_\mu(x)}^t f_\lambda(\tau)\,d\tau \\
&= F_\lambda(t)+\big[f_\lambda(u_\mu(x))u_\mu(x)-F_\lambda(u_\mu(x))\big],
\end{align*}
hence for all $u\in V\cap\w$ (note that $u>u_\mu$ in $\Omega$)
\[\hat\Phi_\lambda(u) = \frac{\|u\|^p}{p}-\int_\Omega F_\lambda(u)\,dx-\int_\Omega\big[f_\lambda(u_\mu)u_\mu-F_\lambda(u_\mu)\big]\,dx = \Phi_\lambda(u)-C,\]
with $C\in\R$ independent of $u$. So, recalling that $u_\lambda$ minimizes $\hat\Phi_\lambda$ over $\w$, for all $u\in V\cap\w$ we have
\[\Phi_\lambda(u) \ge \Phi_\lambda(u_\lambda),\]
i.e., $u_\lambda$ is a local minimizer of $\Phi_\lambda$ in $\cs$. By Proposition \ref{svh}, $u_\lambda$ is as well a local minimizer of $\Phi_\lambda$ in $\w$. To proceed with the proof we need to perform a different truncation on the reaction. Set for all $(x,t)\in\Omega\times\R$
\[\tilde f_\lambda(x,t) = \begin{cases}
f_\lambda(t) & \text{if $t\le u_\lambda(x)$} \\
\lambda u_\lambda^{q-1}(x)-t^{r-1} & \text{if $t>u_\lambda(x)$}
\end{cases}\]
and as usual
\[\tilde F_\lambda(x,t) = \int_0^t \tilde f_\lambda(x,\tau)\,d\tau.\]
Clearly $\tilde f_\lambda:\Omega\times\R\to\R$ satisfies ${\bf H}$. So, we set for all $u\in\w$
\[\tilde\Phi_\lambda(u) = \frac{\|u\|^p}{p}-\int_\Omega\tilde F_\lambda(x,u)\,dx\]
and thus define a functional $\tilde\Phi_\lambda\in C^1(\w)$. We note that for all $(x,t)\in\Omega\times\R$ we have $\tilde f_\lambda(x,t)\le f_\lambda(t)$ and hence $\tilde F_\lambda(x,t) \le F_\lambda(t)$. This in turn implies for all $u\in\w$
\beq\label{two1}
\tilde\Phi_\lambda(u) \ge \Phi_\lambda(u).
\eeq
Since $u_\lambda$ is a local minimizer of $\Phi_\lambda$, we can find $\rho>0$ s.t.\ $\Phi_\lambda(u)\ge\Phi_\lambda(u_\lambda)$ for all $u\in B_\rho(u_\lambda)$, hence by \eqref{two1}
\[\tilde\Phi_\lambda(u) \ge \Phi_\lambda(u) \ge \Phi_\lambda(u_\lambda) = \tilde\Phi_\lambda(u_\lambda).\]
So, $u_\lambda$ is as well a local minimizer of $\tilde\Phi_\lambda$. Besides, fix $\eps\in(0,\hat\lambda_1)$ (with $\hat\lambda_1>0$ defined by \eqref{pev}), then we can find $\delta>0$ s.t.\ for all $x\in\R$, $|t|\le\delta$
\[\tilde F_\lambda(x,t) \le F_\lambda(t) \le \eps\frac{(t^+)^p}{p}.\]
Since $\Omega$ is bounded, we can find $\sigma>0$ s.t.\ $\|u\|_\infty\le\delta$ for all $u\in\cs$, $\|u\|_{0,s}\le \sigma$. Then, using also \eqref{pev}, for all $u\in\w\cap\cs$ with $0<\|u\|_{0,s}\le\sigma$ we have
\[\tilde\Phi_\lambda(u) \ge \frac{\|u\|^p}{p}-\int_\Omega \eps\frac{(u^+)^p}{p}\,dx \ge (\hat\lambda_1-\eps)\frac{\|u\|_p^p}{p} >0.\]
So, $0$ is a strict local minimizer of $\tilde\Phi_\lambda$ in $\cs$. By Proposition \ref{svh} again, $0$ is as well a local minimizer of $\tilde\Phi_\lambda$ in $\w$. From Lemma \ref{exi} we know that $\Phi_\lambda$ is coercive in $\w$, so by \eqref{two1} $\tilde\Phi_\lambda$ is coercive as well. As recalled in Section \ref{sec2}, $\tilde\Phi_\lambda$ then satisfies the $(PS)$-condition. Thus, we may apply the mountain pass theorem (see \cite[Theorem 2.1]{PS}) and deduce the existence of $v_\lambda\in\w\setminus\{0,u_\lambda\}$ s.t.\ $\tilde\Phi_\lambda'(v_\lambda)=0$ in $W^{-s,p'}(\Omega)$. So we have weakly in $\Omega$
\beq\label{two2}
\fpl v_\lambda = \tilde f_\lambda(x,v_\lambda).
\eeq
Testing \eqref{two2} with $-v_\lambda^-\in\w$ and applying \eqref{pnp} we have
\[\|v_\lambda^-\|^p \le \langle\fpl v_\lambda,-v_\lambda^-\rangle = \int_\Omega \tilde f_\lambda(x,v_\lambda)(-v_\lambda^-)\,dx = 0,\]
so $v_\lambda\in\w_+\setminus\{0\}$. Recalling the definition of $\tilde f_\lambda$ and testing \eqref{two2} with $(v_\lambda-u_\lambda)^+\in\w$, we have
\begin{align*}
\langle\fpl v_\lambda,(v_\lambda-u_\lambda)^+\rangle &= \int_\Omega \tilde f_\lambda(x,v_\lambda)(v_\lambda-u_\lambda)^+\,dx \\
&\le \int_\Omega f_\lambda(u_\lambda)(v_\lambda-u_\lambda)^+\,dx \\
&= \langle\fpl u_\lambda,(v_\lambda-u_\lambda)^+\rangle,
\end{align*}
which by Proposition \ref{stm} implies $v_\lambda\le u_\lambda$ in $\Omega$. So, \eqref{two2} rephrases as
\[\fpl v_\lambda = f_\lambda(v_\lambda)\]
weakly in $\Omega$. Using Theorem \ref{smp} as in Theorem \ref{sub}, we see that $v_\lambda\in{\rm int}(\cs_+)$ and it solves $(P_\lambda)$. So we have $v_\lambda\le u_\lambda$ in $\Omega$, $v_\lambda\not\equiv u_\lambda$, and weakly in $\Omega$
\[\fpl v_\lambda+v_\lambda^{r-1} = \lambda v_\lambda^{q-1} \le \lambda u_\lambda^{q-1} = \fpl u_\lambda+u_\lambda^{r-1}.\]
By Theorem \ref{scp} we have $u_\lambda-v_\lambda\in{\rm int}(\cs_+)$.
\end{proof}

\noindent
To complete the picture, we examine the limiting case $\lambda=\lambda_*$. In such case we can prove existence of at least one solution, to which all principal solutions $u_\lambda$ converge:

\begin{lemma}\label{str}
There exists a solution $u_*\in{\rm int}(\cs_+)$ of $(P_{\lambda_*})$ s.t.\ $u_\lambda\to u_*$ in both $\w$ and $\cs$ as $\lambda\to\lambda_*^+$.
\end{lemma}
\begin{proof}
We prove a slightly more general assertion. Let $(\lambda_n)$ be a decreasing sequence s.t.\ $\lambda_n\to\lambda_*^+$, and denote by $u_n\in{\rm int}(\cs_+)$ any solution of $(P_{\lambda_n})$, then up to a subsequence $u_n\to u_*$ in both $\w$ and $\cs$ as $n\to\infty$, being $u_*\in{\rm int}(\cs_+)$ a solution of $(P_{\lambda_*})$. First, for all $n\in\N$ we have weakly in $\Omega$
\beq\label{str1}
\fpl u_n = f_{\lambda_n}(u_n).
\eeq
Arguing as in the proof of Theorem \ref{sub}, we find $u_*\in\w_+$ s.t.\ up to a subsequence $u_n\to u_*$ in both $\w$ and $\cs$, hence we can pass to the limit in \eqref{str1} and get weakly in $\Omega$
\beq\label{str2}
\fpl u_* = f_{\lambda_*}(u_*).
\eeq
We claim that $u_*\not\equiv 0$. Arguing by contradiction, assume that $u_n\to 0$ in both $\w$ and $\cs$, hence in particular $u_n\to 0$ uniformly in $\Omega$. Then, for all $n\in\N$ big enough we have $0< u_n\le 1$ in $\Omega$. Set for all $n\in\N$
\[v_n = \frac{u_n}{\|u_n\|} \in\w\cap{\rm int}(\cs_+).\]
The sequence $(v_n)$ is obviously bounded in $\w$. By reflexivity and the compact embedding $\w\hookrightarrow L^p(\Omega)$, passing to a subsequence we have $v_n\rightharpoonup v$ in $\w$, $v_n\to v$ in $L^p(\Omega)$. Besides, by \eqref{str1}, for all $n\in\N$ we have weakly in $\Omega$
\beq\label{str3}
\fpl v_n = \lambda_n\frac{u_n^{q-1}}{\|u_n\|^{p-1}}-\frac{u_n^{r-1}}{\|u_n\|^{p-1}}.
\eeq
Consider the first term in the right hand side of \eqref{str3}. Since $0<u_n\le 1$ in $\Omega$ and $p<q$, we have
\[0 < \frac{u_n^{q-1}}{\|u_n\|^{p-1}} \le \frac{u_n^{p-1}}{\|u_n\|^{p-1}} = v_n^{p-1},\]
so $(u_n^{q-1}/\|u_n\|^{p-1})$ is bounded in $L^{p'}(\Omega)$. Passing to a subsequence, we have $u_n^{q-1}/\|u_n\|^{p-1}\rightharpoonup w$ in $L^{p'}(\Omega)$, hence {\em a fortiori} in $L^1(\Omega)$. By H\"older's inequality and the continuous embedding $\w\hookrightarrow L^q(\Omega)$ we have
\begin{align*}
\|w\|_1 &\le \liminf_n\int_\Omega\frac{u_n^{q-1}}{\|u_n\|^{p-1}}\,dx \\
&\le \limsup_n\frac{\|u_n\|_q^{q-1}|\Omega|^\frac{1}{q}}{\|u_n\|^{p-1}} \\
&\le C\limsup_n\|u_n\|^{q-p} = 0.
\end{align*}
So we get $w=0$, i.e.,
\beq\label{str4}
\frac{u_n^{q-1}}{\|u_n\|^{p-1}}\rightharpoonup 0 \ \text{in $L^{p'}(\Omega)$.}
\eeq
An entirely similar argument proves that $(u_n^{r-1}/\|u_n\|^{p-1})$ is bounded in $L^{p'}(\Omega)$ and, up to a subsequence,
\beq\label{str5}
\frac{u_n^{r-1}}{\|u_n\|^{p-1}}\rightharpoonup 0 \ \text{in $L^{p'}(\Omega)$.}
\eeq
Testing \eqref{str3} with $(v_n-v)\in\w$ and using H\"older's inequality, we have for all $n\in\N$
\begin{align*}
\langle\fpl v_n,v_n-v\rangle &= \int_\Omega\Big[\lambda_n\frac{u_n^{q-1}}{\|u_n\|^{p-1}}-\frac{u_n^{r-1}}{\|u_n\|^{p-1}}\Big](v_n-v)\,dx \\
&\le \lambda_1\Big\|\frac{u_n^{q-1}}{\|u_n\|^{p-1}}\Big\|_{p'}\|v_n-v\|_p-\Big\|\frac{u_n^{r-1}}{\|u_n\|^{p-1}}\Big\|_{p'}\|v_n-v\|_p,
\end{align*}
and the latter tends to $0$ as $n\to\infty$ by the relations above. By the $(S)_+$-property of $\fpl$ we have $v_n\to v$ in $\w$, hence $\|v\|=1$. On the other hand, testing \eqref{str3} with $v\in\w$, we have for all $n\in\N$
\[\langle\fpl v_n,v\rangle = \int_\Omega\Big[\lambda_n\frac{u_n^{q-1}}{\|u_n\|^{p-1}}-\frac{u_n^{r-1}}{\|u_n\|^{p-1}}\Big]v\,dx.\]
Passing to the limit as $n\to\infty$ and recalling \eqref{str4} and \eqref{str5} we get $\|v\|^p=0$, a contradiction. Summarizing, $u_*\in\w_+\setminus\{0\}$ and satisfies \eqref{str2}. As in Lemma \ref{exi} we see that $u_*\in{\rm int}(\cs_+)$ solves $(P_{\lambda_*})$.
\vskip2pt
\noindent
Finally, taking into account the monotonicity property of Lemma \ref{mon}, we conclude that globally $u_\lambda\to u_*$ in both $\w$ and $\cs$, with monotone convergence, as $\lambda\to\lambda_*^+$, for some $u_*\in{\rm int}(\cs_+)$ solving $(P_{\lambda_*})$.
\end{proof}

\noindent
Looking at the proof of Lemma \ref{str} above, we can easily argue that, for any sequence $(\lambda_n)$ s.t.\ $\lambda_n\to\lambda_*^+$, the sequence of solutions $(v_{\lambda_n})$ provided by Lemma \ref{two} has a subsequence which converges to a solution of $(P_{\lambda_*})$, which might differ from the global limit of $u_\lambda$.
\vskip2pt
\noindent
Combining Lemmas \ref{exi}--\ref{str}, we obtain the following bifurcation result for the superdiffusive case (corresponding to case $(c)$ of Theorem \ref{main}):

\begin{theorem}\label{sup}
Let $2\le p<q<r<\p$. Then, there exists $\lambda_*>0$ with the following properties: for all $\lambda\in(0,\lambda_*)$ problem $(P_\lambda)$ has no solution; $(P_{\lambda_*})$ has at least one solution $u_*\in{\rm int}(\cs_+)$; and for all $\lambda>\lambda_*$ problem $(P_\lambda)$ has at least two solutions $u_\lambda,v_\lambda\in{\rm int}(\cs_+)$ s.t.\ $u_\lambda-v_\lambda\in{\rm int}(\cs_+)$, $u_\lambda-u_\mu\in{\rm int}(\cs_+)$ for all $\lambda>\mu>\lambda_*$, and $u_\lambda\to u_*$ in both $\w$ and $\cs$ as $\lambda\to\lambda_*$.
\end{theorem}

\begin{remark}\label{gen}
For simplicity, we confined our study to the pure power logistic reactions. Nevertheless, most of our Theorem \ref{sup} can be extended to the following generalized logistic equation:
\[\begin{cases}
\fpl u = \lambda f(x,u)-g(x,u) & \text{in $\Omega$} \\
u > 0 & \text{in $\Omega$} \\
u = 0 & \text{in $\Omega^c$,}
\end{cases}\]
where $f,g:\Omega\times\R\to\R$ are Carath\'eodory mappings, both $(p-1)$-superlinear at $\infty$ and at $0$, satisfying a subcritical growth condition like ${\bf H}$, and jointly satisfying a pseudo-monotonicity condition (see \cite{IP} for the case of the $p$-Laplacian).
\end{remark}

\vskip4pt
\noindent
{\small {\bf Acknowledgement.} A.\ Iannizzotto and S.\ Mosconi are members of GNAMPA (Gruppo Nazionale per l'Analisi Matematica, la Probabilit\`a e le loro Applicazioni) of INdAM (Istituto Nazionale di Alta Matematica 'Francesco Severi') and are supported by the grant PRIN n.\ 2017AYM8XW: {\em Non-linear Differential Problems via Variational, Topological and Set-valued Methods}. A.\ Iannizzotto is also supported by the research project {\em Evolutive and Stationary Partial Differential Equations with a Focus on Biomathematics} (Fondazione di Sardegna 2019). S.\ Mosconi is also supported by grant PdR 2020-2022, linea 2: MOSAIC and linea 3: PERITO of the University of Catania. We wish to thank S.\ Jarohs for a stimulating discussion on the strong comparison principle.}

\end{document}